\documentclass[11pt,reqno]{amsart}

\usepackage{amsfonts, amsthm, amsmath}
\allowdisplaybreaks[4]

\usepackage{rotating}

\usepackage{tikz}

\usepackage{graphics}

\usepackage{amssymb}

\usepackage{circledsteps}

\usepackage{amsmath}

\usepackage{amscd}

\usepackage[latin2]{inputenc}

\usepackage{t1enc}

\usepackage[mathscr]{eucal}

\usepackage{indentfirst}

\usepackage{graphicx}

\usepackage{graphics}

\usepackage{pict2e}

\usepackage{mathrsfs}

\usepackage{enumerate}

\usepackage[pagebackref]{hyperref}
\hypersetup{colorlinks=true}
\usepackage{cite}
\usepackage{color}
\usepackage{epic}
\usepackage{hyperref} 
\usepackage{framed}
\usepackage{mathabx}

\numberwithin{equation}{section}
\theoremstyle{plain}

\newtheorem{theorem}{Theorem}[section]

\newtheorem{lemma}[theorem]{Lemma}

\newtheorem{corollary}[theorem]{Corollary}

\newtheorem{proposition}[theorem]{Proposition}

\theoremstyle{definition}

\newtheorem{conj}[theorem]{Conjecture}

\newtheorem{remark}[theorem]{Remark}

\newtheorem{?}[theorem]{Problem}

\newtheoremstyle{named}{}{}{\itshape}{}{\bfseries}{.}{.5em}{#1\thmnote{ #3}}
\theoremstyle{named}

\newcommand{\f}[1]{\ifthenelse{\equal{#1}{1}}{(q;q)_\infty}{(q^{#1};q^{#1})_{\infty}}}

\def\bZ{\mathbb{Z}}

\def\bC{\mathbb{C}}
\def\bN{\mathbb{N}}
\def\cp{\mathrm{cp}}
\def\lcm{\mathrm{lcm}}

\def\la{\lambda}

\begin{document}
\title[Further results on non-negativity conjectures for copartition products]{Further results on non-negativity conjectures for copartition products}

\author[H. Li]{Haijun Li}
\address[Haijun Li]{College of Mathematics and Statistics, Chongqing University, Chongqing 401331, P.R. China}
\email{lihaijun@cqu.edu.cn; lihaijune@163.com}

\date{\today}

\begin{abstract}
Burson and Eichhorn proposed infinite and finite coefficientwise non-negativity conjectures for products arising from copartitions. We obtain several further results. First, both products admit decompositions into elementary local blocks, and the finite conjecture reduces exactly to its diagonal specialization. The finite conjecture is proved when the first truncation parameter is one or when the two residue parameters are equal. We also prove that every finite product has non-negative coefficients below the third shifted signed exponent. For fixed finite parameters, the coefficient sequence is shown to be eventually quasipolynomial; an explicit period bound, an explicit onset for the quasipolynomial formula, and the leading term on every residue class are obtained. This yields eventual strict positivity, apart from forced odd-degree zeros in one solved diagonal case. For the infinite conjecture, we prove non-negativity under the broader condition $a\equiv b\pmod m$, which includes the known case $a=b$. Finally, for arbitrary $b\mid a$, we improve the uniform initial non-negative range from degrees below $a+m$ to degrees below $a+2m$.
\end{abstract}


\maketitle
\section{Introduction}\label{sec:intro}

Throughout this paper, we adopt the following standard $q$-series notation; see, for example,~\cite{GasperRahman}.
\begin{align*}
&(a; q)_n := \prod_{k=0}^{n-1} (1 - aq^k),\quad  (a; q)_\infty := \prod_{k=0}^{\infty} (1 - aq^k), \quad |q| < 1,\\
&(a_1, \ldots, a_m; q)_n := (a_1; q)_n \cdots (a_m; q)_n, \qquad n \in \mathbb{N}_0 \cup \{\infty\},
\end{align*}
where we denote the set of positive integers as $\bN$ and define $\bN_0=\bN\cup\{0\}$. 

For a given integer $n\in \bN$, a {\it partition} $\la$ of $n$ is a non-increasing list of positive integer that sum up to $n$. Standard references for partitions include \cite{AndrewsBook}. There has been extensive research on various generalizations of ordinary integer partitions, such as overpartitions~\cite{CorteelLovejoy}, colored partitions~\cite{AlladiGordon}, cylindric partitions~\cite{Borodin, BridgesUncu}, and others. On the other hand, researchers have investigated the non-negativity (or positivity) of a large number of $q$-series sequences; see, for example, \cite{Liu, Stanton, Zhou} and the references therein.

In \cite{BursonEichhornCopartitions}, Burson and Eichhorn introduced $(a, b, m)$-copartitions which reframe and generalize Andrews' $\mathcal{E}\mathcal{O}^{*}$-type partitions~\cite{AndrewsEO}. For $a, b, m\in \bN$, each $(a, b, m)$-copartition is comprised of three partitions: (i) a partition into parts $\equiv a\ (\mathrm{mod}\ m)$, which we call the ground; (ii) a partition into parts $\equiv b\ (\mathrm{mod}\ m)$, which we call the sky; and (iii) a rectangular partition that unites them. We denote the number of $(a, b, m)$-copartitions of $n$ as $\cp_{a, b, m}(n)$.

In \cite{BursonEichhornPositivity}, Burson and Eichhorn proved that the product 
\begin{equation}\label{eq:F-definition}
F_{a,b,m}(q):=\frac{(-q^{a+b};q^m)_\infty}
{(-q^a;q^m)_\infty(q^b;q^m)_\infty}
\end{equation}
is the generating function for the difference between the numbers of $(a, b, m)$-copartitions having an even and an odd number of ground parts. More precisely, let $\cp^o_{a, b, m}(n)$ (resp. $\cp^e_{a, b, m}(n)$) be the number of $(a, b, m)$-copartitions of $n$ with an odd (resp. even) number of ground parts, then
\[
F_{a,b,m}(q)
=
\sum_{n\geq0}
\bigl(\cp^{e}_{a,b,m}(n)
-
\cp^{o}_{a,b,m}(n)\bigr)q^n.
\]
This led them to the following conjecture.

\begin{conj}[{c.f. \cite[Conjecture 5.1]{BursonEichhornPositivity}}]\label{conj:infinite}
For $a,b,m\in\bN$, if $b\mid a$, then
\[
F_{a,b,m}(q)\in\bN_0[[q]].
\]
\end{conj}
Burson and Eichhorn proved Conjecture~\ref{conj:infinite} when $a=b$ \cite[Theorem 5.2]{BursonEichhornPositivity}. Craig's non-negativity theorem for $(q,-q^3;q^4)_\infty^{-1}$ implies the case $(a,b,m)=(3,1,4)$ after multiplication by $(-q^4;q^4)_\infty$; see \cite{Craig}. 

On the other hand, the finite product introduced in \cite{BursonEichhornPositivity} is
\begin{equation}\label{eq:g-definition}
g_{a,b,m}(N,M;q)
:=
\frac{(-q^m;q^m)_{N+M-1}}
{(-q^a;q^m)_N(q^b;q^m)_M}.
\end{equation}
Regarding this finite product, Burson and Eichhorn also proposed a conjecture on the non-negativity of its coefficients.

\begin{conj}[{c.f. \cite[Conjecture 5.3]{BursonEichhornPositivity}}]\label{conj:finite}
For $a, b, m, N, M\in \bN$, if $b\mid a$, $a+b=m$ and $N\leq M$, then 
$$
g_{a, b, m}(N, M, q)\in \bN_0[[q]].
$$
\end{conj}

We note that neither a complete proof nor a counterexample is available for Conjectures~\ref{conj:infinite} and \ref{conj:finite}. The purpose of this paper is to develop the following partial results. For the finite product \eqref{eq:g-definition}, we set $a=kb$ and $x=q^b$ for $k\in \bN$. Then we have $m=(k+1)b$, and 
\begin{equation}\label{eq:G-intro}
G_{k,N,M}(x)
:=
g_{a,b,a+b}(N,M;q)
=
\frac{(-x^{k+1};x^{k+1})_{N+M-1}}
{(-x^{k};x^{k+1})_N(x;x^{k+1})_M}.
\end{equation}
Moreover, we define
\[
D_{k,N}(x):=G_{k,N,N}(x).
\]
Our first result is that, for determining the non-negativity of the coefficients of $G_{k,N,M}(x)$, it suffices to consider the case $N=M$, namely $D_{k,N}(x)$.

\begin{theorem}\label{thm:diagonal-intro}
For $k, N, M\in \bN$ and $N\leq M$, we have
\begin{equation}\label{eq:diagonal-factor-intro}
G_{k,N,M}(x)
=
D_{k,N}(x)
\prod_{j=N}^{M-1}
\frac{1+x^{(k+1)(N+j)}}{1-x^{(k+1)j+1}}.
\end{equation}
Consequently, Conjecture~\ref{conj:finite} is equivalent to
\[
D_{k, N}(x)\in\bN_0[[x]].
\]
\end{theorem}

The next result is to show that, over a determinable initial range, the coefficients of the finite product \eqref{eq:g-definition} under the hypotheses of Conjecture \ref{conj:finite} are non-nagative.

\begin{theorem}\label{thm:finite-initial-intro}
Under the hypotheses of Conjecture~\ref{conj:finite}, we have
\[
[q^n]g_{a,b,a+b}(N,M;q)\geq0
\qquad
(0\leq n<a+2(a+b)).
\]
Equivalently,
\[
[x^r]G_{k,N,M}(x)\geq0
\qquad
(0\leq r<3k+2).
\]
\end{theorem}

For fixed finite parameters in $G_{k, N, M}(x)$, we can observe the behavior of the coefficients more precisely.

\begin{theorem}\label{thm:quasipoly-intro}
For fixed $k, N, M\in\bN$ with $N\leq M$, we write
\[
G_{k,N,M}(x)=\sum_{n\geq0}c_{k,N,M}(n)x^n,
\]
and define
\[
\Delta_{k,N,M}:=(k+1)N(M-1)+N-M
\]
and
\[
\Lambda_{k,N,M}
:=
\lcm\Bigl(
\left\{(k+1)j+1:0\leq j<M\right\}
\cup
\{2(k+1)i-2:1\leq i\leq N\}
\Bigr).
\]
There exist polynomials
\[
Q_0(X),Q_1(X),\ldots,Q_{\Lambda_{k,N,M}-1}(X)
\in\mathbb{Q}[X]
\]
of degree at most $M-1$ such that
\[
c_{k,N,M}(n)=Q_r(n)
\]
whenever
\[
n>\Delta_{k,N,M}
\qquad\text{and}\qquad
n\equiv r\pmod{\Lambda_{k,N,M}}.
\]

\begin{itemize}
\item If $M>N$, every $Q_r$ has leading coefficient
\[
\frac{2^{M-1}}
{(M-1)!\prod_{j=0}^{M-1}((k+1)j+1)}.
\]

\item If $M=N\geq2$ and $k$ is even, every $Q_r$ has leading coefficient
\[
\frac{2^{N-1}}
{(N-1)!\prod_{j=0}^{N-1}((k+1)j+1)}.
\]

\item If $M=N\geq2$ and $k>1$ is odd, the leading coefficient in parity class $r$ is
\[
\frac{1}{(N-1)!}
\left(
\frac{2^{N-1}}{\prod_{j=0}^{N-1}((k+1)j+1)}
+
(-1)^r
\frac{2^{N-1}}{\prod_{i=1}^{N}((k+1)i-1)}
\right),
\]
and this number is positive for both parities.
\end{itemize}
Consequently, every fixed finite product is eventually strictly positive, except that when $M=N$ and $k=1$, every odd-degree coefficient is zero and every even-degree coefficient is positive.
\end{theorem}

On the other hand, for infinite product \eqref{eq:F-definition} we next give a complete infinite-product subfamily. It strictly includes the known case $a=b$.

\begin{theorem}\label{thm:shift-aligned-intro}
Let $a, b, m\in\bN$ with $a\geq b$. If $a\equiv b\pmod m$, then $F_{a,b,m}(q)\in\bN_0[[q]]$. More precisely, if $a=b+tm$, then
\begin{equation}\label{eq:shift-factor-intro}
F_{a,b,m}(q)
=
\frac{(-q^{a+b};q^m)_\infty}
{(q^b;q^m)_t(q^{2a};q^{2m})_\infty}\in\bN_0[[q]].
\end{equation}
\end{theorem}

For the remaining infinite products, we can improve the guaranteed initial range.

\begin{theorem}\label{thm:infinite-initial-intro}
For $a, b, m\in \bN$, if $b\mid a$, then
\[
[q^n]F_{a,b,m}(q)\geq0
\qquad
(0\leq n<a+2m).
\]
\end{theorem}

The paper is organized as follows. Section~\ref{sec:pre} develops the local block and the finite and infinite decompositions. Section~\ref{sec:finite} and \ref{sec:14} prove Theorems \ref{thm:diagonal-intro} and \ref{thm:finite-initial-intro}, respectively. Section~\ref{sec:quasipoly} establishes Theorem \ref{thm:quasipoly-intro} through a root-of-unity pole analysis. Section~\ref{sec:infinite} proves Theorems \ref{thm:shift-aligned-intro} and \ref{thm:infinite-initial-intro}. Section~\ref{sec:conclusion} treats the balanced limit and summarizes the remaining obstruction.

\section{Preliminaries}\label{sec:pre}

We begin with the elementary identity that underlies several later arguments.

\begin{theorem}\label{thm:local-block-intro}
For $A,B\in\bN$, define
\[
R_{A,B}(q)
:=
\frac{1+q^{A+B}}{(1+q^A)(1-q^B)},
\]
then we have
\begin{equation}\label{eq:local-identity-intro}
R_{A,B}(q)
=
\frac{1}{1-q^B}-\frac{q^A}{1+q^A}
=
\frac{1}{1-q^B}+
\sum_{r\geq1}(-1)^r q^{rA}.
\end{equation}
Moreover,
\[
R_{A,B}(q)\in\bN_0[[q]]
\quad\Longleftrightarrow\quad
B\mid A.
\]
\end{theorem}

\begin{proof}
A direct calculation gives 
\begin{align*}
\frac{1}{1-q^B}-\frac{q^A}{1+q^A}=R_{A, B}(q).
\end{align*}
Also, using $$-\frac{q^A}{1+q^A}=-q^A+q^{2A}-q^{3A}+\cdots=\sum_{r\geq1}(-1)^r q^{rA}$$ proves \eqref{eq:local-identity-intro}.

Suppose first that $B\nmid A$. In the series
\[
\frac{1}{1-q^B}
=
1+q^B+q^{2B}+\cdots,
\]
there is no term of degree $A$. The second series in \eqref{eq:local-identity-intro} contributes $-q^A$. Hence
\[
[q^A]R_{A,B}(q)=-1,
\]
so $R_{A,B}(q)$ is not coefficientwise non-negative.

Now suppose that $B\mid A$, we can write $A=uB$. At a degree $n$ that is not divisible by $B$, both series in \eqref{eq:local-identity-intro} have coefficient zero. At a degree $n=sB$, the geometric series contributes $1$. An additional contribution occurs only if $sB=rA=ruB$, or equivalently $s=ru$. In that case the total coefficient is
\[
1+(-1)^r,
\]
which is either $0$ or $2$. Thus every coefficient is non-negative. This completes the proof.
\end{proof}

Both the finite and infinite copartition products admit exact decompositions into products of those local blocks in Theorem \ref{thm:local-block-intro} and additional coefficientwise non-negative factors.

\subsection{Infinite decomposition} For $j\geq 0$, we set 
\begin{align*}
A_j:=a+jm,\quad B_j:=b+jm,\quad \text{and }C_j:=a+b+jm.
\end{align*}

\begin{proposition}\label{prop:infinite-local}
The product $F_{a,b,m}(q)$ has the coefficientwise formal factorization
\begin{equation}\label{eq:infinite-local-factorization}
F_{a,b,m}(q)
=
\prod_{j\geq0}R_{A_j,B_j}(q)
\prod_{j\geq0}
\bigl(1+q^{C_{2j+1}}\bigr).
\end{equation}
\end{proposition}

\begin{proof}
We split the numerator of $F_{a, b, m}(q)$ into its even- and odd-indexed factors:
\begin{align*}
(-q^{a+b};q^m)_\infty
&=
\prod_{r\geq0}(1+q^{C_r})\\
&=
\prod_{j\geq0}(1+q^{C_{2j}})
\prod_{j\geq0}(1+q^{C_{2j+1}}).
\end{align*}
Since $C_{2j}=A_j+B_j$, pairing $1+q^{C_{2j}}$ with
\[
(1+q^{A_j})(1-q^{B_j})
\]
in the denominator of $F_{a, b, m}(q)$ produces $R_{A_j,B_j}(q)$. Multiplying over all $j$ proves \eqref{eq:infinite-local-factorization}.
\end{proof}

\begin{remark}
When $b\mid a$, the zeroth block is non-negative because $B_0=b$ divides $A_0=a$. For $a>b$, however, the divisibility $B_j\mid A_j$ cannot hold for every sufficiently large $j$, since
\[
A_j-B_j=a-b
\]
is fixed while $B_j\to\infty$. Thus the Conjecture \ref{conj:infinite} cannot be proved by showing that every local block in \eqref{eq:infinite-local-factorization} is non-negative.
\end{remark}

\subsection{Finite decomposition} 

Assume now that $m=a+b$. For $0\leq j<N$, we have
\[
A_j+B_j
=a+b+2jm
=(2j+1)m.
\]
The numerator factor of exponent $(2j+1)m$ can therefore be paired with the $j$th signed and ordinary denominator factors.

Define
\[
\mathcal{E}_{N,M}
:=
\{1,2,\ldots,N+M-1\}
\setminus
\{1,3,5,\ldots,2N-1\}.
\]
The inclusion $2N-1\leq N+M-1$ follows from $N\leq M$.

\begin{proposition}\label{prop:finite-local}
Under the hypotheses of Conjecture~\ref{conj:finite}, we have
\begin{equation}\label{eq:finite-local-factorization}
\begin{split}
g_{a,b,m}(N,M;q)
={}
\prod_{j=0}^{N-1}R_{A_j,B_j}(q)\times
\prod_{r\in\mathcal{E}_{N,M}}(1+q^{rm})
\prod_{j=N}^{M-1}\frac{1}{1-q^{B_j}}.
\end{split}
\end{equation}
\end{proposition}

\begin{proof}
We can write out the product:
\[
g_{a,b,m}(N,M;q)
=
\frac{\prod_{r=1}^{N+M-1}(1+q^{rm})}
{\prod_{j=0}^{N-1}(1+q^{A_j})
 \prod_{j=0}^{M-1}(1-q^{B_j})}.
\]
For $0\leq j<N$, we pair the numerator factor $1+q^{(2j+1)m}$ with
\[
(1+q^{A_j})(1-q^{B_j}).
\]
The resulting quotient is $R_{A_j,B_j}(q)$. The unpaired numerator indices form $\mathcal{E}_{N,M}$, while the unpaired ordinary denominator factors have $N\leq j<M$. This proves \eqref{eq:finite-local-factorization}.
\end{proof}

\begin{remark}
Let $a=kb$ and $m=(k+1)b$. If $k>1$, then
\[
A_1=(2k+1)b,
\qquad
B_1=(k+2)b.
\]
The divisibility $B_1\mid A_1$ would require
\[
k+2\mid 2k+1.
\]
Since
\[
2k+1=2(k+2)-3,
\]
this would imply $k+2\mid3$, which is impossible for $k>1$. By Theorem~\ref{thm:local-block-intro}, the block $R_{A_1,B_1}$ has coefficient $-1$ at degree $A_1$. Therefore positivity of the complete finite product for $k>1$ and $N\geq2$ must result from cancellation between different factors.
\end{remark}

\section{Proof of Theorem \ref{thm:diagonal-intro}}\label{sec:finite}

\begin{proof}[Proof of Theorem \ref{thm:diagonal-intro}]
By definition, we can write out the factors:
\begin{equation}\label{eq:G-expanded}
G_{k,N,M}(x)
=
\frac{\prod_{r=1}^{N+M-1}(1+x^{(k+1)r})}
{
 \prod_{i=1}^{N}(1+x^{(k+1)i-1})
 \prod_{j=0}^{M-1}(1-x^{(k+1)j+1})}.
\end{equation}
Then we can extract from $G_{k,N,M}(x)$ all the factors that constitute $D_{k, N}=G_{k,N,N}(x)$:
\begin{align*}
G_{k, N, M}(x)&=\left[\frac{\prod_{r=1}^{2N-1}(1+x^{(k+1)r})}
{
 \prod_{i=1}^{N}(1+x^{(k+1)i-1})
 \prod_{j=0}^{N-1}(1-x^{(k+1)j+1})}\right]\times \left[\frac{\prod_{r=2N}^{N+M-1}(1+x^{(k+1)r})}{\prod_{j=N}^{M-1}(1-x^{(k+1)j+1})}\right]\\
 &=D_{k, N}(x)\frac{\prod_{r=2N}^{N+M-1}(1+x^{(k+1)r})}{\prod_{j=N}^{M-1}(1-x^{(k+1)j+1})}\\
 &=D_{k, N}(x)\prod_{j=N}^{M-1}
\frac{1+x^{(k+1)(N+j)}}{1-x^{(k+1)j+1}}.\quad\text{(by $r=N+j$)}
\end{align*}
Every factor in the additional product has non-negative coefficients. Hence non-negativity of every diagonal product $D_{k, N}(x)$ implies non-negativity of every product with $M\geq N$.

Conversely, Conjecture~\ref{conj:finite} includes the diagonal choice $M=N$. Thus the conjecture is equivalent to non-negativity of all $D_{k,N}(x)$.
\end{proof}
Thus we obtain some results in special cases.

\subsection{The case $N=1$}

\begin{theorem}\label{thm:N-one-intro}
For all $k,M\in\bN$,
\[
G_{k,1,M}(x)\in\bN_0[[x]].
\]
Equivalently, Conjecture~\ref{conj:finite} holds whenever $N=1$.
\end{theorem}

\begin{proof}
By Theorem \ref{thm:diagonal-intro}, it suffices to consider
\[
D_{k,1}(x)
=
\frac{1+x^{k+1}}{(1+x^k)(1-x)}.
\]
Using
\[
\frac{1}{1+x^k}
=
\frac{1-x^k}{1-x^{2k}},
\]
we obtain
\begin{align*}
D_{k,1}(x)
&=
\frac{(1-x^k)(1+x^{k+1})}
{(1-x)(1-x^{2k})}=
\frac{(1+x+x^2+\cdots+x^{k-1})(1+x^{k+1})}
{1-x^{2k}}.
\end{align*}
The last expression is a product of a polynomial with non-negative coefficients and a geometric series. Hence, we have
\[
D_{k,1}(x)\in\bN_0[[x]].
\]
Theorem~\ref{thm:diagonal-intro} now gives the result for every $M\in\bN$.
\end{proof}

\subsection{The case $a=b$}
\begin{theorem}\label{thm:equal-finite-intro}
For all $a,N,M\in\bN$ with $N\leq M$, we have
\[
\frac{(-q^{2a};q^{2a})_{N+M-1}}
{(-q^a;q^{2a})_N(q^a;q^{2a})_M}
\in\bN_0[[q]].
\]
Equivalently, Conjecture~\ref{conj:finite} holds whenever $a=b$.
\end{theorem}

\begin{proof}
When $a=b$, we have $k=1$. By Theorem \ref{thm:diagonal-intro}, it suffices to prove non-negativity of
\[
D_{1,N}(x)
=
\frac{\prod_{r=1}^{2N-1}(1+x^{2r})}
{\prod_{i=1}^{N}(1+x^{2i-1})
 \prod_{j=0}^{N-1}(1-x^{2j+1})}.
\]
For each $i=1,\ldots,N$, we pair the denominator factors having exponent $2i-1$:
\[
(1+x^{2i-1})(1-x^{2i-1})
=1-x^{4i-2}.
\]
The numerator contains $1+x^{4i-2}$, corresponding to the odd numerator index $r=2i-1$. The even numerator indices are $r=2j$ for $1\leq j<N$, and their factors are $1+x^{4j}$. Therefore
\begin{equation}\label{eq:d-two-positive-factorization}
D_{1,N}(x)
=
\prod_{i=1}^{N}
\frac{1+x^{4i-2}}{1-x^{4i-2}}
\prod_{j=1}^{N-1}(1+x^{4j}).
\end{equation}
Every factor on the right has non-negative coefficients. This proves non-negativity of $D_{1,N}(x)$ and hence, by Theorem~\ref{thm:diagonal-intro}, of every $G_{1,N,M}(x)$ with $M\geq N$.
\end{proof}

\begin{corollary}\label{cor:d-two-support}
For every $N\in \bN$, we have
\[
[x^{2r+1}]D_{1,N}(x)=0
\qquad(r\in\bN_0),
\]
whereas every even-degree coefficient is positive.
\end{corollary}

\begin{proof}
Formula \eqref{eq:d-two-positive-factorization} involves only even powers of $x$, proving the first assertion. Its first factor with $i=1$ contains $(1-x^2)^{-1}$, while every remaining factor has constant term $1$. Hence every even power occurs with positive coefficient.
\end{proof}

\section{Proof of Theorem \ref{thm:finite-initial-intro}}\label{sec:14}

For $k\in\bN$, we have
\[
D_{k, 1}(x)
=
\frac{1+x^{k+1}}{(1+x^k)(1-x)}
=
\sum_{r\geq0}\ell_k(r)x^r.
\]
The local identity of Theorem~\ref{thm:local-block-intro} gives
\begin{equation}\label{eq:Lk-sparse}
D_{k, 1}(x)
=
\frac{1}{1-x}
+
\sum_{s\geq1}(-1)^s x^{sk}.
\end{equation}

\begin{lemma}\label{lem:ell-values}
For $r\geq0$, we have
\[
\ell_k(r)
=
\begin{cases}
1, & r=0,\\
1, & k\nmid r,\\
0, & r=(2s+1)k\text{ for some }s\geq0,\\
2, & r=2sk\text{ for some }s\geq1.
\end{cases}
\]
If we set $\ell_k(r)=0$ for $r<0$, then
\begin{equation}\label{eq:ell-shift-inequality}
\ell_k(r+k-1)+\ell_k(r-1)
\geq
\ell_k(r)
\qquad(r\geq0).
\end{equation}
\end{lemma}

\begin{proof}
The coefficient description follows immediately from \eqref{eq:Lk-sparse}. The geometric series contributes $1$ in every degree. If $r=sk$ with $s\geq1$, the sparse correction contributes $(-1)^s$. Thus the total is $0$ when $s$ is odd and $2$ when $s$ is even.

Now it suffices to prove \eqref{eq:ell-shift-inequality}. If $k=1$, then its left side contains $\ell_1(r+k-1)=\ell_1(r)$, so the inequality is immediate. Next we assume $k\geq2$. For $r=0$, we have
\[
\ell_k(k-1)=1=\ell_k(0).
\]
Now let $r\geq1$. If $\ell_k(r)=0$, there is nothing to prove. If $\ell_k(r)=2$, then $r=2sk$ for some $s\geq1$. Both $r-1$ and $r+k-1$ are congruent to $-1$ modulo $k$, so neither is divisible by $k$. Hence
\[
\ell_k(r-1)=\ell_k(r+k-1)=1,
\]
and their sum equals $2=\ell_k(r)$.

It remains to consider $\ell_k(r)=1$. If $\ell_k(r+k-1)\geq1$, the result is immediate. Otherwise, we have
\[
r+k-1=(2s+1)k
\]
for some $s\geq0$. This implies that
\[
r-1=2sk.
\]
If $s=0$, then $r-1=0$ and $\ell_k(r-1)=1$. If $s\geq1$, then $\ell_k(r-1)=2$. In either case the left side of \eqref{eq:ell-shift-inequality} is at least $1=\ell_k(r)$.
\end{proof}

\begin{proof}[Proof of Theorem~\ref{thm:finite-initial-intro}]
The case $N=1$ follows from Theorem~\ref{thm:N-one-intro}, so we assume $N\geq2$. Then naturally $M\geq N\geq 2$. Let
\[
A_1:=a+m,
\qquad
A_2:=a+2m,
\qquad
B_1:=b+m.
\]
Next we separate the zeroth numerator, signed-denominator, and ordinary-denominator factors in \eqref{eq:g-definition}. Since $m=a+b$, the resulting initial block is
\[
\frac{1+q^m}{(1+q^a)(1-q^b)}
=D_{k, 1}(q^b),
\qquad k=\frac{a}{b}.
\]
By Lemma \ref{lem:ell-values} we define the coefficientwise non-negative series
\begin{equation}\label{eq:finite-positive-H}
H_{N,M}(q)
:=
D_{k, 1}(q^b)
\frac{\prod_{r=2}^{N+M-1}(1+q^{rm})}
{\prod_{j=1}^{M-1}(1-q^{b+jm})}.
\end{equation}
Then
\begin{equation}\label{eq:finite-H-signed-tail}
g_{a,b,m}(N,M;q)
=
H_{N,M}(q)
\prod_{i=1}^{N-1}\frac{1}{1+q^{a+im}}.
\end{equation}
We write
\[
H_{N,M}(q)=\sum_{n\geq0}h(n)q^n.
\]
The signed tail in \eqref{eq:finite-H-signed-tail} is
\[
\prod_{i=1}^{N-1}\frac{1}{1+q^{a+im}}
\equiv
1-q^{A_1}
\pmod{q^{A_2}\bZ[[q]]},
\]
since the next signed factor starts in degree $A_2$, while $2A_1=2a+2m>A_2$. It follows immediately that the coefficients of $g_{a, b, m}(N, M; q)$ below $A_1$ are non-negative. For
\[
n=A_1+t,
\qquad
0\leq t<m,
\]
we have
\begin{equation}\label{eq:finite-difference}
[q^n]g_{a,b,m}(N,M;q)
=h(A_1+t)-h(t).
\end{equation}
We will prove that the right side of \eqref{eq:finite-difference} is non-negative for $0\leq t<m$.

Every non-constant factor in \eqref{eq:finite-positive-H}, apart from $D_{k, 1}(q^b)$, has degree at least $B_1=b+m>m$. Therefore, for $0\leq t<m$, we have
\[
h(t)=[q^t]D_{k, 1}(q^b).
\]
If $b\nmid t$, then $h(t)=0$, and \eqref{eq:finite-difference} is non-negative because $h(A_1+t)\geq0$.

Suppose now that $t=rb$. Then
\[
h(t)=\ell_k(r).
\]
We exhibit two disjoint classes of contributions to $h(A_1+rb)$.

\begin{itemize}
\item First, we choose one copy of the ordinary part $B_1=b+m$ from the factor $(1-q^{B_1})^{-1}$, take constant terms from every other factor outside $D_{k, 1}(q^b)$, and take degree
\begin{align*}
A_1+rb-B_1=(r+k-1)b
\end{align*}
from $D_{k, 1}(q^b)$. These choices contribute $\ell_k(r+k-1)$.

\item Second, we choose the numerator term $q^{2m}$ from the factor $1+q^{2m}$, take constant terms from the other exterior factors, and take degree
\begin{align*}
A_1+rb-2m=(r-1)b
\end{align*}
from $D_{k, 1}(q^b)$, where we used $m=a+b$. This contributes $\ell_k(r-1)$, with the convention that it is zero when $r=0$.
\end{itemize}
The two classes are disjoint because one selects an ordinary denominator part of size $B_1$, while the other selects the numerator part $2m$. Hence
\[
h(A_1+rb)
\geq
\ell_k(r+k-1)+\ell_k(r-1).
\]
Lemma~\ref{lem:ell-values} gives
\[
h(A_1+rb)
\geq
\ell_k(r)
=h(rb).
\]
Equation \eqref{eq:finite-difference} is therefore non-negative for every $0\leq t<m$. This proves non-negativity for
\[
0\leq n<A_2=a+2m.
\]
Since $m=a+b$, this is the stated range $n<a+2(a+b)$.
\end{proof}

\section{Proof of Theorem \ref{thm:quasipoly-intro}}\label{sec:quasipoly}

The coefficients of a rational function whose poles are roots of unity are eventually quasipolynomial. This is standard; see Stanley \cite[Chapter 4]{Stanley} and Flajolet--Sedgewick \cite[Chapter IV]{FlajoletSedgewick}. We record the precise form needed here.

\begin{lemma}\label{lem:root-unity-pf}
Let $R(x)\in\bC(x)$ be regular at $x=0$, and suppose every finite pole of $R$ is a root of unity. If a pole $\zeta$ has order $s_\zeta$, then the partial-fraction decomposition
\[
R(x)
=P(x)
+
\sum_{\zeta}
\sum_{r=1}^{s_\zeta}
\frac{c_{\zeta,r}}{(1-x/\zeta)^r}
\]
for a polynomial $P(x)$. Moreover,
\[
[x^n]\frac{1}{(1-x/\zeta)^r}
=
\zeta^{-n}\binom{n+r-1}{r-1}.
\]
Consequently, once $n>\deg P$, the coefficient $[x^n]R(x)$ is a quasipolynomial in $n$. A pole of order $s$ contributes a term of degree at most $s-1$.
\end{lemma}

\subsection{Degree and period bounds}

From \eqref{eq:G-expanded}, the numerator degree is
\[
\frac{(k+1)(N+M-1)(N+M)}{2},
\]
and the denominator degree is
\begin{align*}
\sum_{i=1}^{N}((k+1)i-1)
+
\sum_{j=0}^{M-1}((k+1)j+1)
=
\frac{(k+1)N(N+1)}{2}-N+
\frac{(k+1)M(M-1)}{2}+M.
\end{align*}
Subtracting the two degrees above yields
\begin{equation}\label{eq:degree-difference}
\Delta_{k,N,M}
=(k+1)N(M-1)+N-M.
\end{equation}
Thus the polynomial part in the partial-fraction decomposition of \eqref{eq:G-expanded} has degree at most $\Delta_{k,N,M}$.

Every root of $1-x^{(k+1)j+1}$ has order dividing $(k+1)j+1$. Every root of $1+x^{(k+1)i-1}$ has order dividing $2((k+1)i-1)$. Hence all poles have orders dividing
\[
\Lambda_{k,N,M}
=
\lcm\Bigl(
\{(k+1)j+1:0\leq j<M\}
\cup
\{2((k+1)i-1):1\leq i\leq N\}
\Bigr).
\]
Lemma~\ref{lem:root-unity-pf} therefore proves the existence, period bound, degree bound, and onset asserted in Theorem~\ref{thm:quasipoly-intro}. It remains to determine the maximal pole orders and their leading constants.

\subsection{Pole orders}

Let $\zeta$ be a primitive $t$th root of unity. Define
\begin{align*}
U_t
&:=
\#\{1\leq i\leq N:1+\zeta^{(k+1)i-1}=0\},\\
V_t
&:=
\#\{0\leq j\leq M-1:1-\zeta^{(k+1)j+1}=0\},\\
W_t
&:=
\#\{1\leq r\leq N+M-1:1+\zeta^{(k+1)r}=0\}.
\end{align*}
Every vanishing factor has a simple zero. Thus the pole order at $\zeta$, if positive, is
\begin{equation}\label{eq:pole-order-count}
U_t+V_t-W_t.
\end{equation}

\begin{lemma}\label{lem:maximal-poles}
If we assume $(N,M)\neq (1,1)$, then 
\begin{itemize}
\item[(i).] The point $x=1$ is a pole of order exactly $M$.
\item[(ii).] If $M>N$, every root of unity other than $1$ is a pole of order at most $M-1$.
\item[(iii).] If $M=N$, every root of unity other than $1$ has pole order at most $N-1$, except that $x=-1$ has pole order exactly $N$ when $k$ is odd.
\end{itemize}
\end{lemma}

\begin{proof}
At $x=1$, all $M$ factors $1-x^{(k+1)j+1}$ vanish, while every factor of the form $1+x^{(k+1)i-1}$ or $1+x^{(k+1)r}$ equals $2$. Thus $x=1$ has pole order exactly $M$. Now let $t>1$ and $h:=\gcd(k+1,t)$. We split the argument into two cases.

\noindent\textbf{Case 1: $h=1$.} The condition counted by $V_t$ is
\[
(k+1)j\equiv-1\pmod t.
\]
Since $k+1$ is invertible modulo $t$, this selects one residue class modulo $t$. Consequently, we have
\begin{equation}\label{eq:V-bound}
V_t\leq\left\lceil\frac{M}{t}\right\rceil.
\end{equation}
We distinguish between two subcases according to the parity of $t$.
\begin{description}
\item[Case 1-1] If $t$ is odd, then no power of the primitive $t$th root $\zeta$ equals $-1$. Therefore, $U_t=W_t=0$. Since $t\geq3$ and $M\geq2$, except in the excluded case $(N,M)=(1,1)$, we have
\[
U_t+V_t-W_t
=V_t
\leq
\left\lceil\frac{M}{t}\right\rceil
\leq M-1.
\]

\item[Case 1-2] If $t$ is even, then the conditions counted by $U_t$ and $W_t$ are
\[
(k+1)i\equiv1+t/2\pmod t,\text{ and }(k+1)r\equiv t/2\pmod t,
\]
respectively. Define
\begin{align*}
I&:=\{1\leq i\leq N:(k+1)i\equiv1+t/2\pmod t\},\\
J&:=\{0\leq j\leq M-1:(k+1)j\equiv-1\pmod t\}.
\end{align*}
Thus $|I|=U_t$ and $|J|=V_t$. If $i\in I$ and $j\in J$, then
\[
(k+1)(i+j)\equiv t/2\pmod t.
\]
Also,
\[
1\leq i+j\leq N+M-1,
\]
so $i+j$ is counted by $W_t$.

\begin{itemize}
\item If both $I$ and $J$ are nonempty, they are arithmetic progressions with common difference $t$. Their sumset $\{i+j: i\in I\text{ and }j\in J\}$ contains at least $|I|+|J|-1=U_t+V_t-1$ distinct integers. Every one is counted by $W_t$. Hence, $W_t\geq U_t+V_t-1$, and therefore $U_t+V_t-W_t\leq1$. Since $M\geq2$, this is at most $M-1$; in the diagonal case $N=M\geq2$, it is at most $N-1$.

\item If $I$ is empty, then
\[
U_t+V_t-W_t\leq V_t
\leq\left\lceil\frac{M}{t}\right\rceil
\leq M-1.
\]

\item If $J$ is empty, the same residue-class argument gives
\[
U_t+V_t-W_t\leq U_t
\leq\left\lceil\frac{N}{t}\right\rceil.
\]
This is at most $N-1$ when $N\geq2$, and at most $M-1$ when $N=1<M$.
\end{itemize}
\end{description}

\noindent\textbf{Case 2: $h>1$.} The congruence
\[
(k+1)j\equiv-1\pmod t
\]
has no solution because the left side is divisible by $h$ and the right side is not. Thus, we have $V_t=0$.  We distinguish between two subcases according to the parity of $t$.

\begin{description}
\item[Case 2-1] If $t$ is odd, then $U_t=W_t=0$ as before, and there is no pole.

\item[Case 2-2] If $t$ is even, then the congruence defining $U_t$ can be solvable only if $h\mid1+t/2$, whereas the congruence defining $W_t$ can be solvable only if $h\mid t/2$. These two conditions cannot both hold, since their difference is $1$.

If the second condition holds, then $U_t=0$ and the quantity \eqref{eq:pole-order-count} is non-positive. If neither holds, there is no pole. We are left with $h\mid1+t/2$. After division by $h$, the solutions counted by $U_t$ form one residue class modulo $L:=t/h$. 

\begin{itemize}
\item If $L\geq2$, then we have
\[
U_t\leq\left\lceil\frac{N}{L}\right\rceil
\leq\left\lceil\frac{N}{2}\right\rceil.
\]
For $N\geq2$, this is at most $N-1$. If $N=1$, the excluded case forces $M\geq2$, and $U_t\leq1\leq M-1$.

\item It remains to consider $L=1$, then we have $t=h$ and $t\mid (k+1)$. The condition $h\mid1+t/2$ becomes $t\mid1+t/2$.
Since $t$ is even, this forces $t=2$. Thus the only exceptional root is $x=-1$, and it can occur only when $k$ is odd.

For odd $k$, every exponent $(k+1)i-1$ is odd, so all $N$ signed denominator factors vanish at $-1$. Every exponent $(k+1)j+1$ is odd, so the ordinary denominator factors equal $2$ there. Every numerator exponent $(k+1)r$ is even, so the numerator factors also equal $2$. Hence $-1$ is a pole of order exactly $N$.
\end{itemize}
\end{description}
If $M>N$, this order is at most $M-1$. If $M=N$, it is exactly the second possible maximal pole order. This completes the proof.
\end{proof}

\subsection{Leading constants}

We define
\[
A^+_{k,M}
:=
\frac{2^{M-1}}
{\prod_{j=0}^{M-1}((k+1)j+1)}.
\]
And for odd $k$, we define
\[
A^-_{k,N}
:=
\frac{2^{N-1}}
{\prod_{i=1}^{N}((k+1)i-1)}.
\]

\begin{lemma}\label{lem:pole-constants}
We have
\[
\lim_{x\to1}(1-x)^M G_{k,N,M}(x)
=A^+_{k,M}.
\]
If $M=N$ and $k$ is odd, then
\[
\lim_{x\to-1}(1+x)^N D_{k,N}(x)
=A^-_{k,N}.
\]
\end{lemma}

\begin{proof}
For $e\geq1$, we have
\[
\lim_{x\to1}\frac{1-x^e}{1-x}=e.
\]
At $x=1$, the $N+M-1$ numerator factors of $G_{k, N, M}(x)$ contribute $2^{N+M-1}$, while the $N$ signed denominator factors contribute $2^N$. The remaining ordinary factors contribute
\[
\prod_{j=0}^{M-1}((k+1)j+1)
\]
after multiplication by $(1-x)^M$. This proves the first formula.

Now suppose $M=N$ and $k$ is odd. At $x=-1$, all $2N-1$ numerator factors tend to $2$, and all $N$ ordinary denominator factors tend to $2$. For odd $e$, we have
\[
\lim_{x\to-1}\frac{1+x^e}{1+x}=e.
\]
The signed denominator exponents are $(k+1)i-1$. Hence
\[
\lim_{x\to-1}(1+x)^N D_{k,N}(x)
=
\frac{2^{2N-1}}
{2^N\prod_{i=1}^{N}((k+1)i-1)}
=A^-_{k,N}.
\]
This complete the proof.
\end{proof}

\subsection{Complete proof and remarks}

\begin{proof}[Proof of Theorem~\ref{thm:quasipoly-intro}]
The existence and exact onset of the quasipolynomial representation follow from Lemma~\ref{lem:root-unity-pf} and the degree calculation \eqref{eq:degree-difference}. The period divides $\Lambda_{k,N,M}$ because every pole is a $\Lambda_{k,N,M}$th root of unity.

\begin{itemize}
\item Suppose that $M>N$. By Lemma~\ref{lem:maximal-poles}, $x=1$ is the unique pole of order $M$. Lemma~\ref{lem:pole-constants} shows that its leading partial-fraction term is
\[
\frac{A^+_{k,M}}{(1-x)^M}.
\]
Since
\[
[x^n](1-x)^{-M}
=
\binom{n+M-1}{M-1}
=
\frac{n^{M-1}}{(M-1)!}+O(n^{M-2}),
\]
every residue-class polynomial has leading coefficient
\[
\frac{A^+_{k,M}}{(M-1)!}
=
\frac{2^{M-1}}
{(M-1)!\prod_{j=0}^{M-1}((k+1)j+1)},
\]
which is positive.

\item Suppose that $M=N\geq2$ and $k$ is even. Again, $x=1$ is the unique pole of maximal order, now equal to $N$. The same argument gives the stated leading coefficient
\[
\frac{A^+_{k,N}}{(N-1)!}=\frac{2^{N-1}}
{(N-1)!\prod_{j=0}^{N-1}((k+1)j+1)}.
\]

\item Suppose that $M=N\geq2$ and $k>1$ is odd. The only poles of order $N$ are $1$ and $-1$. Their leading terms are
\[
\frac{A^+_{k,N}}{(1-x)^N}
\qquad\text{and}\qquad
\frac{A^-_{k,N}}{(1+x)^N}.
\]
Because
\[
[x^n](1+x)^{-N}
=(-1)^n\binom{n+N-1}{N-1},
\]
the leading coefficient in parity class $r$ is
\[
\frac{A^+_{k,N}+(-1)^rA^-_{k,N}}{(N-1)!}.
\]
We must prove that this is positive for odd as well as even $r$. Indeed, we can rewrite
\[
\prod_{j=0}^{N-1}((k+1)j+1)
=
\prod_{i=1}^{N}\bigl((k+1)(i-1)+1\bigr).
\]
For every $1\leq i\leq N$, we have $(k+1)i-1-\bigl((k+1)(i-1)+1\bigr)=k-1>0$. Therefore, $(k+1)(i-1)+1<(k+1)i-1$ term by term. Hence
\[
\prod_{j=0}^{N-1}((k+1)j+1)
<
\prod_{i=1}^{N}((k+1)i-1),
\]
which implies
\[
A^+_{k,N}>A^-_{k,N}>0.
\]
Thus both $A^+_{k,N}+A^-_{k,N}$ and $A^+_{k,N}-A^-_{k,N}$ are positive.
\end{itemize}
Finally, when $k=1$ on the diagonal, Corollary~\ref{cor:d-two-support} gives the exact conclusion: odd coefficients vanish and even coefficients are positive. The case $N=1$ was proved exactly in Theorem~\ref{thm:N-one-intro}.

Every nonzero residue-class polynomial therefore has positive leading coefficient. It follows that its values are positive for all sufficiently large arguments in that residue class. This completes the proof.
\end{proof}

\begin{corollary}\label{cor:finite-certification}
For every fixed triple $(k,N,M)$ with $N\leq M$, the assertion
\[
G_{k,N,M}(x)\in\bN_0[[x]]
\]
can be decided by a finite exact computation.
\end{corollary}

\begin{proof}
The coefficients through degree $\Delta_{k,N,M}$ form a finite list. Beyond that degree, the sequence is represented by finitely many rational polynomials $Q_r$. These polynomials can be recovered exactly, for example by partial fractions or by interpolation from sufficiently many exact coefficients in each residue class. Theorem~\ref{thm:quasipoly-intro} shows that each nonzero constituent has positive leading coefficient, so each can have only finitely many negative values on nonnegative integers. Exact real-root isolation or direct sign analysis of the finitely many polynomials therefore supplies a finite verification bound.
\end{proof}

\begin{remark}
Corollary~\ref{cor:finite-certification} is a pointwise result. It does not provide a uniform verification bound valid simultaneously for all $k,N,M$, and therefore it does not prove Conjecture~\ref{conj:finite}.
\end{remark}

\section{Proofs of Theorems \ref{thm:shift-aligned-intro} and \ref{thm:infinite-initial-intro}}\label{sec:infinite}

\subsection{Shift-aligned residue classes}

\begin{proof}[Proof of Theorem~\ref{thm:shift-aligned-intro}]
We first assume that $a=b+tm$ for an integer $t\geq0$. The elementary splitting identity $(z;q)_\infty
=(z;q)_t(zq^t;q)_\infty$ and the merging identity $(z; q)_{\infty}(-z; q)_{\infty}=(z^2; q^2)_{\infty}$ give
\begin{align*}
(q^b;q^m)_\infty
=
(q^b;q^m)_t(q^{b+tm};q^m)_\infty=
(q^b;q^m)_t(q^a;q^m)_\infty,
\end{align*}
and 
\begin{align*}
(-q^a;q^m)_\infty(q^a;q^m)_\infty
=(q^{2a};q^{2m})_\infty,
\end{align*}
respectively. Substituting these two identities into \eqref{eq:F-definition} yields
\begin{align*}
F_{a,b,m}(q)
&=
\frac{(-q^{a+b};q^m)_\infty}
{(q^b;q^m)_t
 (-q^a;q^m)_\infty
 (q^a;q^m)_\infty}\\
 &=\frac{(-q^{a+b};q^m)_\infty}
{(q^b;q^m)_t(q^{2a};q^{2m})_\infty},
\end{align*}
which is the claimed form. The numerator is a product of factors $1+q^r$, each of which has non-negative coefficients. The reciprocal of each denominator factor $1-q^s$ is a geometric series with non-negative coefficients. Therefore the complete product lies in $\bN_0[[q]]$.
\end{proof}

\begin{corollary}\label{cor:shift-subfamilies}
Theorem~\ref{thm:shift-aligned-intro} contains each of the following cases:
\begin{itemize}
\item[(i).] $a=b$, for arbitrary $m$;
\item[(ii).] $m\mid b\mid a$;
\item[(iii).] more generally, every case of Conjecture~\ref{conj:infinite} for which $m\mid(a-b)$.
\end{itemize}
\end{corollary}

\begin{proof}
If $a=b$, then $a-b=0$ is divisible by every positive $m$. If $m\mid b\mid a$, then $m$ divides both $a$ and $b$, hence divides $a-b$. The third assertion is exactly the theorem under the conjecture's assumptions.
\end{proof}

\begin{remark}
Factorization \eqref{eq:shift-factor-intro} gives a direct colored-partition interpretation. The numerator permits distinct parts in the progression
\[
a+b,a+b+m,a+b+2m,\ldots.
\]
The finite reciprocal $(q^b;q^m)_t^{-1}$ permits unrestricted parts
\[
b,b+m,\ldots,b+(t-1)m,
\]
and $(q^{2a};q^{2m})_\infty^{-1}$ permits unrestricted parts congruent to $2a$ modulo $2m$. If two listed types have the same numerical size, they are naturally regarded as different colors.
\end{remark}

\subsection{Improving the initial non-negative range}

We assume now that $b\mid a$, and write $a=kb$. Moreover, we set
\[
A_1=a+m,
\qquad
A_2=a+2m,
\qquad
B_1=b+m,
\qquad
C_1=a+b+m.
\]
Separate the zeroth factors in \eqref{eq:F-definition}:
\begin{equation}\label{eq:infinite-separated}
F_{a,b,m}(q)
=
D_{k, 1}(q^b)H(q)
\prod_{j=1}^{\infty}\frac{1}{1+q^{a+jm}},
\end{equation}
where
\begin{equation}\label{eq:infinite-positive-H}
H(q)
:=
\frac{(-q^{a+b+m};q^m)_\infty}
{(q^{b+m};q^m)_\infty}.
\end{equation}
Both $D_{k, 1}(q^b)$ and $H(q)$ have non-negative coefficients.

\begin{proof}[Proof of Theorem~\ref{thm:infinite-initial-intro}]
We write
\[
D_{k, 1}(q^b)H(q)=\sum_{n\geq0}h'(n)q^n.
\]
The signed tail in \eqref{eq:infinite-separated} satisfies
\[
\prod_{j=1}^{\infty}\frac{1}{1+q^{a+jm}}
\equiv
1-q^{A_1}
\pmod{q^{A_2}\bZ[[q]]}.
\]
Indeed, the next signed factor starts at $A_2$, and $2A_1>A_2$. Thus the coefficients below $A_1$ are non-negative. For
\[
n=A_1+t,
\qquad
0\leq t<m,
\]
we have
\begin{equation}\label{eq:infinite-coefficient-difference}
[q^n]F_{a,b,m}(q)
=h'(A_1+t)-h'(t).
\end{equation}
All non-constant terms of $H(q)$ have degree at least $B_1=b+m>m$. Therefore, for $0\leq t<m$, we have
\[
h'(t)=[q^t]D_{k, 1}(q^b).
\]
If $b\nmid t$, this coefficient is zero, so \eqref{eq:infinite-coefficient-difference} is non-negative. Let $t=rb$, then
\[
h'(t)=\ell_k(r).
\]
We again construct two disjoint contributions to $h'(A_1+rb)$.

\begin{itemize}
\item First, select one copy of $q^{B_1}$ from the geometric factor $(1-q^{B_1})^{-1}$ in $H(q)$, select constants from all other factors in $H(q)$, and select degree
\[
A_1+rb-B_1
=(r+k-1)b
\]
from $D_{k, 1}(q^b)$. This contributes $\ell_k(r+k-1)$.

\item Second, select the numerator term $q^{C_1}$ from the factor $1+q^{C_1}$, select constants from all other factors in $H(q)$, and select degree
\[
A_1+rb-C_1
=(r-1)b
\]
from $D_{k, 1}(q^b)$. This contributes $\ell_k(r-1)$.
\end{itemize}
The two classes are disjoint. Hence
\[
h'(A_1+rb)
\geq
\ell_k(r+k-1)+\ell_k(r-1).
\]
By Lemma~\ref{lem:ell-values},
\[
h'(A_1+rb)
\geq
\ell_k(r)
=h'(rb).
\]
Equation \eqref{eq:infinite-coefficient-difference} is therefore non-negative for every $0\leq t<m$. Combining this with the range below $A_1$ proves
\[
[q^n]F_{a,b,m}(q)\geq0
\qquad(0\leq n<A_2=a+2m).
\]
\end{proof}

\begin{remark}
At degree $a+2m$, a second negative signed factor becomes active. The comparison used above then contains two shifted negative contributions, as well as possible positive contributions from repeated use of the first signed factor. Controlling these terms uniformly appears to require a genuinely new injection or a more global product transformation.
\end{remark}

\section{Concluding remarks}\label{sec:conclusion}

Now we record the precise coefficientwise limit of the balanced finite products.

\begin{theorem}\label{thm:limit-intro}
If $b\mid a$ and $m=a+b$, then
\[
\lim_{N\to\infty}g_{a,b,m}(N,N;q)
=
F_{a,b,m}(q)
\]
coefficientwise in $\bZ[[q]]$. Thus Conjecture~\ref{conj:finite}, if proved, would imply the $m=a+b$ subfamily of Conjecture~\ref{conj:infinite}.
\end{theorem}

\begin{proof}
We first assume that $m=a+b$, write $a=kb$, and set $x=q^b$. Then
\[
g_{a,b,m}(N,N;q)
=
\frac{(-x^{k+1};x^{k+1})_{2N-1}}
{(-x^{k};x^{k+1})_N(x;x^{k+1})_N}.
\]
Fix a non-negative integer $L$. Every factor whose first non-constant term has degree greater than $L$ is congruent to $1$ modulo $x^{L+1}$. Therefore, once $N$ is sufficiently large relative to $L$, increasing $N$ does not alter any coefficient of degree at most $L$. This is coefficientwise convergence.

Passing to the limit gives
\[
\lim_{N\to\infty}g_{a,b,m}(N,N;q)
=
\frac{(-x^{k+1};x^{k+1})_\infty}
{(-x^{k};x^{k+1})_\infty(x;x^{k+1})_\infty}.
\]
Since
\[
x^{k+1}=q^m,
\qquad
x^{k}=q^a,
\qquad
x=q^b,
\]
the right side equals
\[
\frac{(-q^m;q^m)_\infty}
{(-q^a;q^m)_\infty(q^b;q^m)_\infty}.
\]
Because $m=a+b$, this is precisely $F_{a,b,m}(q)$.

A coefficientwise limit of series in $\bN_0[[q]]$ has non-negative coefficients. Therefore a proof of Conjecture~\ref{conj:finite} would imply the balanced $m=a+b$ case of Conjecture~\ref{conj:infinite}.
\end{proof}


Finally, we turn to the remaining problem concerning the non-negativity of infinite and finite products.

\subsection{Remaining finite problem}

The diagonal reduction leaves the assertion
\[
D_{k,N}(x)
=
\frac{\prod_{r=1}^{2N-1}(1+x^{(k+1)r})}
{\prod_{i=1}^{N}(1+x^{(k+1)i-1})
 \prod_{j=0}^{N-1}(1-x^{(k+1)j+1})}
\in\bN_0[[x]].
\]
The local decomposition proves that, for $k>1$ and $N\geq2$, at least one local factor has a negative coefficient. Any complete proof must therefore use cancellation between blocks, a sign-reversing involution, or a global analytic identity.


\subsection{Remaining infinite problem}

Theorem~\ref{thm:shift-aligned-intro} settles every case in which the signed and ordinary denominator progressions eventually coincide. Outside this alignment, the local factors contain genuine negative coefficients. Theorem~\ref{thm:infinite-initial-intro} controls the first such negative shift, but the next interval involves several interacting signed choices. A successful proof of Conjecture~\ref{conj:infinite} will likely require either a global copartition injection or an analytic treatment of all relevant roots of unity, analogous in spirit to Craig's circle-method analysis \cite{Craig} but uniform in $a,b,m$.

\section{Declarations}
\subsection{Conflict of interest statement} On behalf of all authors, the corresponding author states that there is no conflict of interest.

\subsection{Data availability} No data was used for the research described in the article.

\subsection{Funding statement} This work received no specific grant from any funding agency in the public, commercial, or not-for-profit sectors.



\end{document}